\documentclass[11pt,a4paper,headinclude,footinclude,fleqn]{scrartcl}                 
\usepackage[T1]{fontenc}                   
\usepackage[utf8]{inputenc}                 
\usepackage[english]{babel}       
\usepackage{graphicx}                      %
\usepackage[font=small]{quoting}            %
\usepackage{caption}        
\usepackage{mathrsfs}          %
\usepackage[nochapters,beramono,eulermath,%
pdfspacing,
listings,
]{classicthesis}                %
\usepackage{arsclassica}                    
\usepackage[top=.8in,bottom=1in,left=1in,right=1in]{geometry}
\usepackage[english]{babel}
\usepackage{stmaryrd}
\usepackage{verbatim}
\usepackage{color}
\usepackage{picture}
\usepackage{amsmath,amsthm,amsfonts,amssymb}
\usepackage{comment}
\usepackage{mathtools}
\usepackage{pgf,tikz}
\usepackage{mathrsfs}
\usetikzlibrary{arrows}

\usepackage{titling} 
\usepackage[hyperpageref]{backref}

\newtheorem{thm}{Theorem}[section]

\newtheorem{prop}[thm]{Proposition}

\theoremstyle{definition}
\newtheorem{rem}[thm]{Remark}
\newtheorem{example}[thm]{Example}
\theoremstyle{remark}

\newcommand{\ds}{\displaystyle}

\newcommand{\R}{\mathbb{R}}

\newcommand{\Rn}{\mathbb R^n}

\newcommand{\de}{\partial}
\newcommand{\eps}{\varepsilon}

\def\O{\Omega}
\newcommand{\ep}{\varepsilon}
\DeclareMathOperator{\dive}{div}
\DeclareMathOperator{\Qp}{\mathcal Q_{p}}

{\left\{\begin{array}{@{}l@{}}}{\end{array}\right.}
\patchcmd{\abstract}{\scshape\abstractname}{\textbf{\abstractname}}{}{}
\makeatletter 
\def\@makefnmark{} 
\makeatother 
\title{ 
On functionals involving the torsional rigidity
related to some classes of nonlinear operators 
}
\author{Francesco Della Pietra, %
Nunzia Gavitone, %
Serena Guarino Lo Bianco \\[.3cm]%
{\em\scriptsize Universit\`a degli studi di Napoli Federico II, Dipartimento di Matematica e Applicazioni ``R. Caccioppoli''} \\ 
{\em\scriptsize Via Cintia, Monte S. Angelo - 80126 Napoli, Italia.}\thanks{Email: f.dellapietra@unina.it, nunzia.gavitone@unina.it, serena.guarinolobianco@unina.it}
}

\begin{document}

\maketitle


\begin{abstract}
\noindent\textsc{Abstract:} In this paper we study optimal estimates for two functionals involving the anisotropic $p$-torsional rigidity $T_p(\Omega)$, $1<p<+\infty$. More precisely, we study
$\Phi(\Omega)=\frac{T_p(\Omega)}{|\Omega|M(\Omega)}$ and
$\Psi(\Omega)=\frac{T_p(\Omega)}{|\Omega|[R_{F}(\Omega)]^{\frac{p}{p-1}}}$,
where $M(\Omega)$ is the maximum of the torsion function $u_{\Omega}$ and $R_F(\Omega)$ is the anisotropic inradius of $\Omega$. 
\medskip

\textsc{Keywords:} torsional rigidity, anisotropic operators, optimal estimates

\medskip
\textsc{MSC 2010:} 49Q10, 35J25

\end{abstract}

\section{Introduction}

Let $F : \mathbb{R}^N \to [0, +\infty[$, $N\ge 2$, be a convex, even,  $1$-homogeneous and $C^{3,\beta}(\mathbb{R}^N\setminus \{0\})$  function such that
$[F^{p}]_{\xi\xi}\text{ is positive definite in } \R^{N}\setminus\{0\}$, $1<p<+\infty$. The anisotropic $p-$laplacian is the 
operator defined by
\[
\mathcal Q_{p}u:=\sum_{i=1}^N \frac{\partial}{\partial x_i} (F(\nabla u)^{p-1}F_{\xi_i}(\nabla u)).
\]
 For $p=2$, $\mathcal Q_2$ is the so-called Finsler Laplacian, while when $F(\xi)= |\xi|$ is the Euclidean norm, $\mathcal Q_p$ reduces to the well known $p$-Laplace operator.

Given a bounded domain $\Omega$ in $\mathbb{R}^{N}$, let us consider the torsion problem for $\mathcal Q_p$:
\begin{equation}
\label{intro:pb_tor}
\left\{
\begin{array}{ll}-\mathcal Q_{p}u=1  &\text{in}\ \Omega \\
u=0 &\text{on}\ \partial\Omega.
\end{array}
\right.
\end{equation}
The anisotropic $p$-torsional rigidity of $\Omega$ is the number $T_p(\Omega)>0$  defined by 
\begin{equation*}
  T_p(\Omega)=\int_\Omega F(\nabla u_{\Omega})^p dx = \int_\Omega u_{\Omega} dx,
\end{equation*}
where $u_{\Omega}\in W_0^{1,p}(\Omega)$ is the torsion function, that is the unique solution of \eqref{intro:pb_tor}.

The main aim of the paper is the study of optimal estimates for the following two functionals involving $T_p(\Omega)$:
\[
\Phi(\Omega)=\frac{T_p(\Omega)}{|\Omega|M(\Omega)},\qquad
\Psi(\Omega)=\frac{T_p(\Omega)}{|\Omega|[R_{F}(\Omega)]^{q}}.
\]
Here and after we will denote by $q$ the H\"older conjugate of $p$, $q=\frac{p}{p-1}$, by $M(\Omega)$ the maximum of the torsion function $u_{\Omega}$ and by $R_F(\Omega)$ the anisotropic inradius of $\Omega$ (see Section \ref{notprel} for the precise definitions).
Observe that the functionals $\Phi$ and $\Psi$ are scaling invariant with respect to the domain. Indeed:
\[
T_p(t\Omega)= t^{N+q}T_p(\Omega),\quad |t\Omega|=t^{N}|\Omega|,\quad M(t\Omega)=t^{q}M(\Omega), \quad R_F(t \Omega)=t R_F(\Omega) .
\]
Our main result is the following.
\begin{thm}
\label{main}
Let $\Omega$ be a convex bounded domain in $\R^{N}$. It holds that
\begin{enumerate}
\item[i)]  $\dfrac{q}{N^{q-1}(N+q)} \le\Phi(\Omega)\le \dfrac{q}{q+1}.$

The right-hand side inequality is optimal for a suitable sequence of thinning rectangles.

\item[ii)] $
\displaystyle \frac{1}{N^{q-1}}\frac{1}{N+q} \le \Psi(\Omega) \le \frac{1}{q+1}.$

The left-hand side inequality holds as an equality if and only if $\Omega$ is a Wulff shape, that is a ball in the dual norm $F^{o}$; the right-hand side inequality is optimal for a suitable sequence of thinning rectangles.
\end{enumerate}
\end{thm}

When $F=\mathcal E$ is the Euclidean norm, there is a wide literature on sharp estimates for $T_{p}(\Omega)$ related to several geometrical quantities depending on $\Omega$. For example, in the classical case of the torsional rigidity for the Laplace operator ($p=2$), with $N=2$, it is known that 
\[
\frac{1}{8} \le \Psi(\Omega)= \frac{T_{2}(\Omega)}{R^{2}_{\mathcal E}|\Omega|}\le \frac{1}{3},
\]
where $R_{\mathcal E}(\Omega)$ is the standard Euclidean inradius of $\Omega$. The left-hand side inequality is due to P\'olya and Szeg\H{o} (see \cite{pz}), while the right-hand side inequality was proved by Makai in \cite{m}. 
As regards the case $p\ne 2$, in \cite{fgl}, among other results, estimates for $\Psi(\Omega)$ are given in the planar case, obtaining  an upper bound and a sharp lower bound. In the anisotropic case, in \cite{bgm} the estimates in \textit{ii)} are proved for $p=2$. 

As regards the functional $\Phi(\Omega)$, up to our knowledge, it seems that the only known result is in the Euclidean case for $p=2$. Indeed, in \cite{hlp} the authors prove the following estimates:
\[
\displaystyle \frac{1}{(N+1)^2}\le\Phi(\Omega)=\frac{T_{2}(\Omega)}{|\Omega|M(\Omega)} \le \frac{2}{3}.
\]
Moreover, they show the optimality of the upper bound, while they conjecture that the lower bound is not optimal, and that the sharp constant in the plane is $\frac 13$, achieved on a sequence of thinning isosceles triangles. In our result, we improve the constant ${(N+1)^{-2}}$, replacing it with $[N(N+2)]^{-1}$. Anyway, we believe that $[N(N+2)]^{-1}$ is not optimal, and for $N=2$ we show that there is a sequence of thinning isosceles triangles $\tau_n$ such that $\Phi(\tau_n)\to \frac 13$.

In order to prove our main result, among the main tools involved, the following estimate for the maximum $M(\Omega)$ of the torsion function $u_{\Omega}$ plays a key role. 

\begin{thm}
\label{intro:sperb}
Let $\Omega$ be a bounded convex domain in $ \R^{N}$, $ N \ge 2$, and let $R_F(\Omega)$ the anisotropic inradius of $\Omega$. Let $u_\Omega$ be the solution of \eqref{intro:pb_tor}. For $1<p< +\infty$ it holds that 
\begin{equation}
\label{intro:stima_max}
\frac{R_{F}^{q}(\Omega)}{qN^{q-1}}\le M(\Omega)\le \frac{R_{F}^{q}(\Omega)}{q}.
\end{equation}
The right-hand side inequality is optimal for a suitable sequence of thinning $N$-rectangular domains. The other inequality, holds as an equality if and only if $\Omega$ is the Wulff shape $\mathcal W_{R}(x_{0})$.
\end{thm}

The upper bound in \eqref{intro:stima_max} has been proved in \cite{pay} in the Euclidean case  for $p=2$, $N=2$ (see also \cite{sp}), by using a $\mathcal P$-function computation and a maximum principle. Anyway, many other estimates for the torsion function are known; the interested reader can refer, for example, to \cite{vdB,bor,hlp} and the reference therein contained. We prove inequality \eqref{intro:stima_max} generalizing the $\mathcal P$-function technique to the case $1<p<+\infty$, and in the anisotropic case. 

Finally, we recall that in the Euclidean case, several other estimates for the $p$-torsional rigidity, involving different geometrical quantities, are known (for the Eucidean case, see for instance \cite{vdBBB,vdBFNT,sp} ($p=2$), \cite{fgl} ($1<p<+\infty$), and \cite{dgmana} for the anisotropic case  ($1<p<+\infty$)). 

The paper is organized as follows. In Section 2 we fix some notation and recall prelimiary results about Finsler metrics  and the anisotropic $p$-torsional rigidity. In Section 3 we prove Theorem \ref{intro:sperb} by using the $\mathcal P$-function method. Finally, in Section 4 we give the proof of the main Theorem \ref{main}. We will split it in  several partial results.



\section{Notation and preliminaries}\label{notprel}
Throughout the paper we will consider a convex even 1-homogeneous function 
\[
\xi\in  \R^{N}\mapsto F(\xi)\in [0,+\infty[,
\] 
that is a convex function such that
\begin{equation}
\label{eq:omo}
F(t\xi)=|t|F(\xi), \quad t\in \R,\,\xi \in  \R^{N}, 
\end{equation}
 and such that
\begin{equation}
\label{eq:lin}
a|\xi| \le F(\xi),\quad \xi \in  \R^{N},
\end{equation}
for some constant $a>0$. The hypotheses on $F$ imply there exists $b\ge a$ such that
\begin{equation}
\label{upb}
F(\xi)\le b |\xi|,\quad \xi \in  \R^{N}.
\end{equation}
Moreover, throughout the paper we will assume that $F\in C^{3,\beta}(\mathbb R^{N}\setminus \{0\})$, and
\begin{equation}
\label{strong}
[F^{p}]_{\xi\xi}(\xi)\text{ is positive definite in } \R^{N}\setminus\{0\},
\end{equation}
with $1<p<+\infty$. 

The hypothesis \eqref{strong} on $F$ ensures that the operator 
\[
\Qp u:= \dive \left(\frac{1}{p}\nabla_{\xi}[F^{p}](\nabla u)\right)
\] 
is elliptic, hence there exists a positive constant $\gamma$ such that
\begin{equation*}
\frac1p\sum_{i,j=1}^{n}{\nabla^{2}_{\xi_{i}\xi_{j}}[F^{p}](\eta)
  \xi_i\xi_j}\ge
\gamma |\eta|^{p-2} |\xi|^2, 
\end{equation*}
for some positive constant $\gamma$, for any $\eta \in
\Rn\setminus\{0\}$ and for any $\xi\in \Rn$. 
\begin{rem}
We stress that for $p\ge 2$ the condition  
\begin{equation*}
\nabla^{2}_{\xi}[F^{2}](\xi)\text{ is positive definite in } \R^{N}\setminus\{0\},
\end{equation*}
implies \eqref{strong}.
\end{rem}
The polar function $F^o\colon \R^N \rightarrow [0,+\infty[$ 
of $F$ is defined as
\begin{equation*}
F^o(v)=\sup_{\xi \ne 0} \frac{\langle \xi, v\rangle}{F(\xi)}. 
\end{equation*}
 It is easy to verify that also $F^o$ is a convex function
which satisfies properties \eqref{eq:omo} and
\eqref{eq:lin}. Furthermore, 
\begin{equation*}
F(v)=\sup_{\xi \ne 0} \frac{\langle \xi, v\rangle}{F^o(\xi)}.
\end{equation*}
From the above property it holds that
\begin{equation}
\label{imp}
|\langle \xi, \eta\rangle| \le F(\xi) F^{o}(\eta), \qquad \forall \xi, \eta \in  \R^{N}.
\end{equation}
The set
\[
\mathcal W = \{  \xi \in  \R^N \colon F^o(\xi)< 1 \}
\]
is the so-called Wulff shape centered at the origin. We put
$\kappa_n=|\mathcal W|$, where $|\mathcal W|$ denotes the Lebesgue measure
of $\mathcal W$. More generally, we denote with $\mathcal W_r(x_0)$
the set $r\mathcal W+x_0$, that is the Wulff shape centered at $x_0$
with measure $\kappa_nr^n$, and $\mathcal W_r(0)=\mathcal W_r$.

We observe that $F$ is the support function  of $\overline{\mathcal W}$. 
In  general  for a nonempty closed convex set $K \subset \R^N$, the support function $h_K$ is defined by
\begin{equation}
\label{suppfunc}
h_K(x):=\sup \{ \langle x, \xi \rangle, \xi \in K\}, \quad \text{ for } x \in \R^N .
\end{equation}
The following properties of $F$ and $F^o$ hold true:
\begin{gather}
\label{prima}
 \langle F_\xi(\xi) , \xi \rangle= F(\xi), \quad  \langle F_\xi^{o} (\xi), \xi \rangle
= F^{o}(\xi),\qquad \forall \xi \in
 \R^N\setminus \{0\}
 \\
 \label{seconda} F(  F_\xi^o(\xi))=F^o(  F_\xi(\xi))=1,\quad \forall \xi \in
 \R^N\setminus \{0\}, 
\\
\label{terza} 
F^o(\xi)   F_\xi( F_\xi^o(\xi) ) = F(\xi) 
 F_\xi^o\left(  F_\xi(\xi) \right) = \xi\qquad \forall \xi \in
 \R^N\setminus \{0\}. 
\end{gather}

\subsection{Anisotropic mean curvature}
Let $\Omega$ be a $C^{2}$ bounded domain, and $\nu(x)$ be the unit outer normal at $x\in\de\Omega$, and let $u\in C^{2}(\overline\Omega)$ such that $\Omega_t=\{u>t\}$, $\de \Omega_t=\{u=t\}$ and $\nabla u\ne 0$ on $\de \Omega_t$.  The anisotropic outer normal $n_{F}$ to $\de \Omega_t$ is given by
  \[
  n_F(x)= F_{\xi}(\nu(x))= F_{\xi}\left(-\nabla u\right),\quad x\in \de \Omega.
  \]
It holds 
 \[
  F^o(n_F)=1.  
  \]
  The anisotropic mean curvature of $\partial \Omega_t$ is  defined as
  \begin{equation}
  \label{H_F} 
  \mathcal H_{F}(x)= \dive\left( n_{F}(x)\right)=
  \dive\left[ \nabla_{\xi}
    F\left(-{\nabla u(x)}\right) \right], \quad x\in
  \de \Omega_t.
  \end{equation}
It holds that
\begin{equation}
\label{der_nf}
\frac{\de u}{\de n_F}= \nabla u\cdot F_{\xi}(-\nabla u)= -F(\nabla u). 
\end{equation}
In \cite{xiath} it has been proved that for a smooth function $u$, on its level sets $\{u=t\}$ it holds
\begin{equation}
\label{xiaformula}
\mathcal Q_2 u=  \frac{\de u}{\de n_F}\mathcal H_F+\frac{\de^2 u}{\de n_F^2},
\end{equation}
where $\frac{\de u}{\de n_F}=\nabla u\cdot n_F$. In the next result we generalize \eqref{xiaformula} for $\mathcal Q_p u$.
\begin{prop}
Let $u$ be a $C^2(\overline\Omega)$ function with a regular level set $\de\Omega_t$. Then we have
\begin{equation}
\label{Q_H}
\mathcal Q_p u=F^{p-2}(\nabla u) \left(  \frac{\de u}{\de n_F}\mathcal H_F+(p-1)\frac{\de^2 u}{\de n_F^2}\right),
\end{equation}
where $\mathcal H_F$ is the anisotropic mean curvature of $\de\Omega_{t}$ as defined in \eqref{H_F}.
\end{prop}
\begin{proof}
 By definition of $\mathcal{Q}_p$, \eqref{xiaformula} and \eqref{der_nf}, we have
\begin{align*}
    \mathcal{Q}_p u&= \dive\left(F^{p-2}(\nabla u) F(\nabla u)F_\xi(\nabla u)\right)\\
    &=
    F^{p-2}(\nabla u) \left( Q_2u +(p-2) F_{\xi_i}(\nabla u)F_{\xi_j}(\nabla u)u_{x_ix_j}\right)\\
    &=F^{p-2}(\nabla u) \left( \frac{\de u}{\de n_F}\mathcal H_F+\frac{\de^2 u}{\de n_F^2} +(p-2) F_{\xi_i}(\nabla u)F_{\xi_j}(\nabla u)u_{x_ix_j}\right)\\
    &=F^{p-2}(\nabla u) \left( \frac{\de u}{\de n_F}\mathcal H_F+(p-1)\frac{\de^2 u}{\de n_F^2}\right),
\end{align*}
that is the thesis.
\end{proof}

Finally we recall the definition of the anisotropic distance from the boundary and the anisotropic inradius.

Let us consider a domain $\Omega$, that is a connected open set of $ \R^N$, with non-empty boundary. 

The anisotropic distance of $x\in\overline\Omega$ to the boundary of $\de \Omega$ is the function 
\begin{equation*}
d_{F}(x)= \inf_{y\in \de \Omega} F^o(x-y), \quad x\in \overline\Omega.
\end{equation*}

We stress that when $F=|\cdot|$ then $d_F=d_{\mathcal{E}}$, the Euclidean distance function from the boundary.

It is not difficult to prove that $d_{F}$ is a uniform Lipschitz function in $\overline \Omega$ and, using the property of $F$ we have
\begin{equation*}
  F(\nabla d_F(x))=1 \quad\text{a.e. in }\Omega.
\end{equation*}
Obviously, assuming $\sup_{\Omega} d_{F}<+\infty$, $d_F\in W_{0}^{1,\infty}(\Omega)$ and the quantity
\begin{equation}
\label{inrad}
R_{F}(\Omega)=\sup \{d_{F}(x),\; x\in\Omega\},
\end{equation}
is called  anisotropic inradius of $\Omega$.

For further properties of the anisotropic distance function we refer
the reader to \cite{cm07}.

\subsection{Anisotropic \texorpdfstring{$p$}{TEXT}-torsional rigidity}
In this subsection we summarize some properties of the anisotropic $p$-torsional rigidity. We refer the reader to \cite{dgmana} for further details.

Let $\Omega$ be a bounded domain in $ \R^{N}$, and $1<p<+\infty$. Throughout the paper we will denote by $q$ the H\"older conjugate of $p$,
\[
q:=\frac{p}{p-1}.
\]
Let us consider the torsion problem for the anisotropic $p-$Laplacian
\begin{equation}
\label{pb_tor}
\left\{
\begin{array}{ll}-\mathcal Q_{p}u:=-\dive \left(F^{p-1}(\nabla u) F_\xi (\nabla u)\right)=1  &\text{in}\ \Omega \\
u=0 &\text{on}\ \partial\Omega.
\end{array}
\right.
\end{equation}
 By classical result there exists  a unique solution of \eqref{pb_tor}, that we will always denote by $u_{\Omega}$, which is positive in $\Omega$. Moreover, by \eqref{strong} and being $F\in C^{3}(\R^{n}\setminus \{0\})$, then $u_{\Omega}\in C^{1,\alpha}(\Omega)\cap C^{3}(\{\nabla u_{\Omega}\ne 0\})$ (see \cite{ladyz,tk84}).

In view of the above considerations, we define the $p$-torsional
anisotropic rigidity of $\Omega$ the number $T_p(\Omega)>0$ such
that  
\begin{equation}
  \label{eq:ptor1}
  T_p(\Omega)=\int_\Omega F(\nabla u_{\Omega})^p dx = \int_\Omega u_{\Omega} dx.
\end{equation}

A characterization of $T_p$ is provided by the equality
$T_p(\Omega)=\sigma(\Omega)^{\frac{1}{p-1}}$, where
$\sigma(\Omega)$ is the best constant in the Sobolev inequality
\[
\|\varphi\|_{L^1(\Omega)}^p \le \sigma(\Omega) \|F(\nabla \varphi)\|^p_{L^p(\Omega)},
\]
that is
\begin{equation}
  \label{tors0}
T_p(\Omega)^{p-1} = \sigma(\Omega)= \max_{\substack{\psi \in
    W_0^{1,p}(\Omega) \setminus\{0\}}}
\dfrac{\left(\displaystyle\int_\Omega |\psi| \, 
    dx\right)^p}{\displaystyle\int_\Omega F(\nabla\psi)^p dx},
\end{equation}
and the solution $u_{\Omega}$ of \eqref{pb_tor} realizes the maximum
in \eqref{tors0}. 

It is immediate to see that if $\Omega\subset\tilde\Omega$, then
\begin{equation}
\label{monotonia}
T_p(\Omega)\le T_p(\tilde\Omega).
\end{equation}
Moreover, by the maximum principle it holds that
\begin{equation}
\label{monotoniamax}
M(\Omega)\le M(\tilde\Omega),
\end{equation}
where $M(\Omega)$ is the maximum of the torsion function in $\Omega$.

A consequence of the anisotropic P\'olya-Szeg\H o inequality (see \cite{aflt}) is the following upper bound for $T_p(\Omega)$ in terms of the measure of $\Omega$.
\begin{thm}
  Let $\Omega$ be a bounded open set of $\mathbb R^N$. Then,
  \begin{equation}
    \label{uptor}
  T_p(\Omega) \le T_p(\mathcal W_R),
  \end{equation} 
  where $\mathcal W_R$ is the Wulff shape centered at the origin with
  the same Lebesgue measure as $\Omega$. 
\end{thm}
\begin{rem}
\label{pb_rad}
If $\Omega = \mathcal W_R$, by the simmetry of the problem, $T_p(\mathcal W_R)$ and the solution $u$ of \eqref{pb_tor} can be explicity calculated. We have:
\begin{equation}
\label{pb_rad2}
u_{\mathcal W}(x)= \frac{R^{q}- F^{o}(x)^q}{qN^{q-1}} \quad\text{and}\quad  T_p(\mathcal W_R)=\frac{1}{N^{q-1}} \frac{|\mathcal W_R|}{N+q}{R^q}.
\end{equation}
\end{rem}
\begin{rem}
We point out that the lower bound in statement \textit{ii)} of Theorem \ref{main} gives a stability type inequality for  \eqref{uptor}. Indeed we have
\[
0 \le T_p(\mathcal W_R)- T_p(\Omega) \le \frac{1}{8} \left( R^2-R^2_F(\Omega)\right),
\]
where $|\mathcal W_R|=|\Omega|.$

\end{rem} 
 
\section{An estimate of the maximum of the torsion function}
In order to give a sharp upper bound  for the maximum $M(\Omega)$ of the torsion function $u_{\Omega}$, we will take into account the following  $\mathcal P$-function:
\begin{equation*}
\mathcal P(x)= \frac{p-1}{p} F^p (\nabla u_{\Omega}) + u_{\Omega} - M(\Omega),
\end{equation*}
where $M(\Omega)=\max_{\Omega} u_{\Omega}$. The following result is proved in \cite{cfv}.
\begin{prop}
\label{cfvprop}
Let $\Omega$ be a domain in $ \R^{N}$, $ N\ge 2$, and $u_{\Omega}\in W_{0}^{1,p}(\Omega)$ be a solution of \eqref{pb_tor}. Set
\[
d_{ij}:= \frac{1}{F(\nabla u_{\Omega})}\de_{\xi_{i}\xi_{j}}\left[\frac{F^{p}}{p}\right](\nabla u_{\Omega}),
\]
Then it holds that
\[
\left( d_{ij}\mathcal P_{i} \right)_{j} -b_{k}\mathcal P_{k} \ge 0 \quad\text{ in }\{\nabla u_{\Omega}\ne 0 \}
\]
where
\[
b_{k}=\frac{p-2}{F^{3}(\nabla u_{\Omega})} F_{\xi_{\ell}}(\nabla u_{\Omega})\mathcal P_{x_{\ell}} F_{\xi_{k}}(\nabla u_{\Omega})+\frac{2p-3}{F^{2}(\nabla u_{\Omega})} \left(  \frac{F_{\xi_{k}\xi_{\ell}}(\nabla u_{\Omega})\mathcal P_{x_{\ell}}}{p-1} -F_{\xi_{k}}(\nabla u_{\Omega})\right)
\]
\end{prop}
As a consequence of the previous result we get the following maximum principle for $\mathcal P$.
\begin{thm}
\label{prsperb}
Let $\Omega$ be a bounded $C^{2}$ domain in $ \R^{N}$, $ N \ge 2$, with nonnegative anisotropic mean curvature $\mathcal H_{F}\ge 0$ on $\de\Omega$, and $u_{\Omega}$ the torsion function. Then 
\begin{equation*}
\frac{p-1}{p}F^p(\nabla u_{\Omega})+u_{\Omega} \le M(\Omega)  \quad \text{ in } \overline{\Omega},
\end{equation*}
that is  the function $\mathcal P$ achieves its maximum at the points $x_M \in \Omega$ such that  $u_{\Omega}(x_M)=M(\Omega)$.
\end{thm}
\begin{proof}
Let us denote by $E_{u_{\Omega}}$ the set of the critical points of $u_{\Omega}$, that is $E_{u_{\Omega}}=\{x \in \overline\Omega \colon \nabla u_{\Omega}(x) =0\}$. Being $\de \Omega$ $C^{2}$, by the Hopf Lemma (see for example \cite{ct}), $E_{u_{\Omega}}\cap \de \Omega=\emptyset$.

Applying Proposition \ref{cfvprop}, the function  $\mathcal P$ verifies a maximum principle in the open set $\Omega \setminus E_{u_{\Omega}}$. Then we have
\[
\max_{\Omega \setminus E_{u_{\Omega}}} \mathcal P = \max_{\de\left(\Omega \setminus E_{u_{\Omega}}\right)}\mathcal P.
\]
Hence one of the following three cases occur
\begin{enumerate}
\item the maximum point of $\mathcal P$ is on $\de \Omega$;
\item the maximum point of $\mathcal P$ is on $E_u$;
\item the function $\mathcal P $ is constant in $\overline \Omega$.
\end{enumerate}
In order to prove the theorem we have to show that  statement 1 cannot happen.
Let us compute the derivative of $\mathcal P$ in the direction of the anisotropic normal $n_F$,  in the sense of  \eqref{der_nf}. Hence we get
\begin{multline}
\label{derivata}
\frac{\de \mathcal P}{\de n_F}= \dfrac{p-1}{p}\frac{\de}{\de n_F}\left(-\frac{\de u_{\Omega}}{\de n_F}\right)^p+\frac{\de u_{\Omega}}{\de n_F} =-(p-1) \left(-\frac{\de u}{\de n_F}\right)^{p-1} \frac{\de^2 u_{\Omega}}{\de n_F^2}+\frac{\de u_{\Omega}}{\de n_F}=\\= -F(\nabla u_{\Omega}) \mathcal Q_p[u]-F^{p-1}(\nabla u_{\Omega}) \mathcal H_F-F(\nabla u_{\Omega})=-F^{p-1}(\nabla u_{\Omega}) \mathcal H_F,
\end{multline}
where last identity follows by \eqref{Q_H}. On the other hand, if a maximum point $\bar x$ of $\mathcal P$ is on $\de\Omega$, by Hopf Lemma either $\mathcal P$ is constant in $\overline\Omega$, or $\frac{\de \mathcal P}{\de n_F}(\bar x)> 0$.
Hence being $\mathcal H_F\ge 0$ we have a contradiction.
\end{proof}

As a consequence of the previous result we get the following optimal estimate for the maximum of $u_{\Omega}$.    
\begin{thm}\label{sperb}
Let $\Omega$ be a bounded convex domain in $ \R^{N}$, $ N \ge 2$, and $1<p< +\infty$. It holds that
\begin{equation}
\label{stima_max}
\frac{R_{F}^{q}(\Omega)}{qN^{q-1}}\le M(\Omega)\le \frac{R_{F}^{q}(\Omega)}{q}.
\end{equation}
\end{thm}
\begin{rem}
In the next section we will show that the right-hand side inequality in \eqref{stima_max} is optimal on a suitable sequence of thinning rectangles (see Proposition \ref{rect} and \eqref{catena1}. We stress that, in general, the quotient $\frac{R^{q}_{F}(\Omega)}{qM(\Omega)}$ approaches the value $1$ also for different sequences of sets (see the example \ref{exHen}).
\end{rem}
\begin{proof}
The left-hand side inequality of \eqref{stima_max} follows by \eqref{monotoniamax} and \eqref{pb_rad2}. Hence, let us prove the other inequality.

First of all, suppose that $\Omega$ is a $C^2$, strictly convex domain. Let $v$ be a direction in $ \R^N$. By Theorem \ref{prsperb} and property \eqref{imp} we have
\begin{equation}
\label{der_dir}
\displaystyle \frac{du_{\Omega}}{dv}=\langle \nabla u, v \rangle  \le F(\nabla u_{\Omega}) F^o(v) \le   \left[ \frac{p}{p-1}\left(M(\Omega)-u_{\Omega}\right)\right]^{\frac{1}{p}}F^o(v),
\end{equation}
where $M(\Omega)$ is the maximum of  $u_{\Omega}$ in $\overline \Omega.$ Let us denote by $x_M$ the point of $\Omega$ such that $M(\Omega )=u_{\Omega}(x_M)$, by $\bar x \in \de \Omega$ such that $F^o(x_M- \bar x)=d_F(x_M)$ and by $v$ the direction of the straight line joining the points $x_m$ and $\bar x$. Then by   \eqref{der_dir} we get
\begin{equation*}
\displaystyle \int_{u(\bar x)}^{M(\Omega)} \displaystyle \frac{1}{\left(M(\Omega)-u_{\Omega}\right)^{\frac{1}{p}}}du \le  \displaystyle  \left( \frac{p}{p-1}\right)^{\frac{1}{p}}F^o(v)|\bar x-x_{M}|= \left( \frac{p}{p-1}\right)^{\frac{1}{p}}F^{o}(\bar x-x_{M}).
\end{equation*}
Being $F^{o}(\bar x-x_{M}) \le R_{F}(\Omega) $, we get
\[
\left( \frac{p}{p-1}\right)M(\Omega)^{ \frac{p-1}{p}}\le  \displaystyle  \left( \frac{p}{p-1}\right)^{\frac{1}{p}}R_{F}(\Omega),
\]
which gives the estimate \eqref{stima_max} for smooth convex domains. 
To prove the estimate in the case of a general convex body $\Omega$, we proceed by approximation. It is well-known (see for example \cite{bf}) that a convex body $\Omega$ can be approximated in the Hausdorff distance by an increasing sequence of smooth strictly convex bodies $\Omega_n\subseteq\Omega$. Clearly, $R_F(\Omega_n)\nnearrow R_F(\Omega)$.

Let $u_n\ge 0$ be the torsion function in $\Omega_n$. In order to conclude the proof we have to show that $M(\Omega_n)\to M(\Omega)$ as $n\to\infty$. We first observe that by \eqref{stima_max}, 
\begin{equation}
\label{maxapp}
M(\Omega_n)\le \frac{R_{F}^{q}(\Omega_n)}{q}\le  \frac{R_{F}^{q}(\Omega)}{q},
\end{equation}
 hence $u_n$ are bounded in $L^\infty(\Omega_n)$. 
 Furthermore, applying Theorem \ref{prsperb} in $\Omega_n$ we have
\[
\frac{p-1}{p}F^p(\nabla u_n)+u_n \le M(\Omega_n)  \quad \text{ in } \overline\Omega_n.
\]
Then by property \eqref{upb}
\begin{equation}
\label{app2}
|\nabla u_n| \le C \quad \text{ in } {\overline\Omega}_n.
\end{equation}
Hence by \eqref{maxapp} and \eqref{app2}, using Ascoli-Arzel\`a theorem we get that $u_n \to u_{\Omega}$ uniformly in $\Omega$ and this  allows to pass to the limit in \eqref{maxapp} and the proof is completed.
\end{proof}
\begin{rem}
We point out that if we take $\Omega$ smooth, the thesis of Theorem \ref{sperb} holds if we assume only that the anisotropic mean curvature of $\Omega$ is nonnegative. 
\end{rem}

\section{Proof of Theorem \ref{main}}
 We split the proof  in various theorems. 
We first prove the lower bound for $\Psi(\Omega)$ in \textit{ii)}.

\begin{thm}\label{torsion}
	If $\Omega\in\mathbb{R}^N$ is a convex bounded domain, $ N \ge 2$, and $1<p<+\infty$, then
	\begin{equation}\label{stimaT}
	\frac{T_p(\Omega)}{|\Omega|} \ge \frac{1}{N^{q-1}} \frac{1}{N+q}{R_F(\Omega)^{q}},
	\end{equation} 
	where $R_F(\Omega)$ is the anisotropic inradius of $\Omega$ defined in \eqref{inrad}. Moreover the equality holds when $\Omega $ is a Wulff shape.
\end{thm}

\begin{proof}
Let us assume first that $\Omega$ is a strictly convex domain and then we remove this assumption with a proof that follows by approximation as in Theorem \ref{sperb}. Let us consider as test function into \eqref{tors0} the following
	\[
	\varphi(x)= \frac{1-\mathcal K^{o}(x)^q}{qN^{q-1}}
	\]
	where $\mathcal K^{o}$ is the support function of the polar set of $\Omega$, defined in  \eqref{suppfunc} . Then $\Omega=\{\mathcal K^{o}<1\}$. By \eqref{pb_rad2}, we observe that when $\Omega =\mathcal W$ then $\varphi$ is exactly the torsion function of the Wulff shape. 
	
	We start computing
	\begin{align*}
	\int_{\Omega} \mathcal K^{o}(x)^q\, dx &= \int_{0}^1 \int_{\mathcal K^o=t} \frac{t^{q}}{|\nabla \mathcal K^o(x)|}\, d\mathcal{H}^{N-1}\, dt = \int_{0}^{1} t^{q} \int_{\mathcal K^o=t} \frac{\mathcal K(\nabla \mathcal K^o(x))}{|\nabla \mathcal K^o|}\, d\mathcal{H}^{N-1}\, dt,\\
	&= \int_{0}^{1} t^{q+N-1}  \int_{\partial\Omega} \mathcal K(\nu_\Omega (x)) \, d\mathcal{H}^{N-1}\, dt = \frac{1}{N+q} \int_{\partial\Omega} x\cdot \nu_{\Omega}(x) \, d\mathcal{H}^{N-1} \\
	&= \frac{1}{N+q} \int_{\Omega} \mbox{div } x \, dx = \frac{N |\Omega|}{N+q},
	\end{align*}
	where $\mathcal K=(\mathcal K^{o})^{o}$.
	Then we have
	\begin{equation}
	\label{uguaglianza}
	\left(\int_{\Omega} \varphi\, dx \right)^p = \frac{|\Omega|^p}{N^q(N+q)^p}=\left(\frac{|\Omega|}{N^{q-1}(N+q)}\right)^p.
	\end{equation}
	Let us now compute 
	\begin{align*}
	\int_{\Omega} F^p(\nabla \varphi) \, dx &= \frac{1}{q^p N^q} \int_{\Omega} F^p\left(\nabla (\mathcal K^o(x)^q)\right)\, dx = 
	\frac{1}{ N^q} \int_{0}^{1}t^{p(q-1)} \int_{\mathcal K^o=t} \frac{F^p(\nabla \mathcal K^o)}{|\nabla \mathcal K^o|}\, d\mathcal{H}^{N-1}\, dt\\[.3cm]
	&
	= \frac{1}{ N^q} \int_{0}^{1} t^{q+N-1} \int_{\partial\Omega} |\nabla \mathcal K^o|^{p-1} F^p(\nu_\Omega)\, d\mathcal{H}^{N-1}\, dt \\
	&=\frac{1}{ N^q} \int_{0}^{1} t^{q+N-1} \int_{\partial\Omega} \frac{F^{p}(\nu_{\Omega})}{\mathcal K^{p-1}(\nu_{\Omega})}\, d\mathcal{H}^{N-1}\, dt,
	\end{align*}
where last equality follows by the identity	$\mathcal K(\nabla \mathcal K^{o}(x))=1$. 
Being $\mathcal W_{R_{F}(\Omega)}\subseteq \Omega$, it follows that $\mathcal K(x) \ge R_{F}(\Omega)F(x)$, so we have
	\begin{multline}
	\label{stima}
	\int_{\Omega} F^p(\nabla \varphi) \, dx \le \frac{1}{(N+q)N^{q}}\cdot \frac{1}{R_{F}(\Omega)^{p}}  \int_{\partial\Omega} \mathcal K(\nu_\Omega (x)) \, d\mathcal{H}^{N-1}=\\ =
	\frac{1}{(N+q)N^{q-1}}\cdot \frac{|\Omega|}{R_{F}(\Omega)^{p}}.
	\end{multline}
	Joining together \eqref{uguaglianza} and \eqref{stima}, we have the thesis.	
	
	Now we prove the validity of \eqref{stimaT} without the assumption on the strict convexity of the domain $\Omega$.
As in the proof of Theorem \ref{sperb}, let $\Omega_n$ be a sequence of smooth strictly convex bodies such that $\Omega_n\to\Omega$. Such a convergence ensures that, as $n\to\infty$,
    \begin{equation}\label{converg}
    |\Omega_n|\to|\Omega|\qquad\hbox{and}\qquad R_F(\Omega_n)\to R_F(\Omega).
    \end{equation}
    By \eqref{monotonia}, it follows that
    \[T_p(\Omega)\ge T_p(\Omega_n),\]
    and by applying \eqref{stimaT} to each $\Omega_n$, we find
\[T_p(\Omega)\ge \frac{{|\Omega_n|}}{N^{q-1}(N+q)}R_F(\Omega_n)^q,\]
which, combined with \eqref{converg}, gives the desired result. Finally we stress that if $\Omega$ is a Wulff shape, the equality case follows from Remark \ref{pb_rad}. On the other hand, if the equality holds in \eqref{stimaT}, then 
	equality must hold in \eqref{stima}, and then $\mathcal K(x)=R_{F}(\Omega)F(x)$, which implies $\Omega=\mathcal W_{R_{F}}$.
	\end{proof}
 Let us consider the functional
\begin{equation}\label{func}
\Phi(\Omega)=\frac{T_p(\Omega)}{|\Omega|M(\Omega)}.
\end{equation}
As  consequence of theorems \ref{torsion} and  \ref{sperb}, we can prove the following estimates for \eqref{func} which is statement i) of Theorem \ref{main}.
\begin{thm}
For any bounded convex domain $\Omega\subset \R^N$, $ N\ge 2$, $1<p<+\infty$ it holds that
\begin{equation}
\label{bounds}
   \frac{q}{N^{q-1}(N+q)} \le\Phi(\Omega)\le \frac{q}{q+1}
\end{equation}
\end{thm}
\begin{proof}
We first prove the lower bound for the functional $\Phi$. By \eqref{stimaT} and \eqref{stima_max} we have
\begin{equation*}
    \Phi(\Omega)= \frac{T_p(\Omega)}{|\Omega|M(\Omega)} \ge \frac{q}{N^{q-1}(N+q)},
\end{equation*}
which gives the lower bound in \eqref{bounds}. 

In order to prove the inequality in the right-hand side in \eqref{bounds}, by Theorem \ref{prsperb} we have
\begin{equation*}
\frac{p-1}{p}F^p(\nabla u_{\Omega})+u_{\Omega} \le M(\Omega) \quad\text{in } \overline \Omega.
\end{equation*}
Integrating in both sides and recalling \eqref{eq:ptor1}, we get
\begin{equation*}
    \left(\frac{p-1}{p}+1\right)T_p(\Omega) \le M(\Omega)|\Omega|,
\end{equation*}
which implies the upper bound in \eqref{bounds}.
\end{proof}
In the following last result we prove the upper bound in statement \textit{ii)} of Theorem \ref{main}, which follows immediately by the preceding results. We stress that in the anisotropic setting, the case $p=2$ was previously considered in \cite{bgm} with a completely different proof.

\begin{thm}\label{torsion_up}
Let $\Omega\in\mathbb{R}^N$ be a bounded convex domain, $ N \ge 2$, $1<p<+\infty$. It holds that
	\begin{equation}\label{stimaT2}
	\frac{T_p(\Omega)}{|\Omega|} \le \frac{R_F(\Omega)^{q}}{q+1}.
	\end{equation} 
\end{thm}
\begin{proof}
By the right-hand side inequality in \eqref{bounds}, and \eqref{stima_max}, we have
\begin{equation}
\label{catena1}
\frac{T_p(\Omega)}{|\Omega|} \le \frac{q}{q+1} M(\Omega)\le 
 \frac{R_{F}^{q} (\Omega)}{q+1}.
\end{equation}
\end{proof}

The final part of the section is devoted to prove the optimality 
of \eqref{stimaT2}. As a consequence, by \eqref{catena1} this will give the optimality of
the right-hand side inequality of \eqref{bounds}, and of \eqref{stima_max}.
\begin{prop}\label{rect}
Let $\O_\varepsilon$ be the $N$-rectangle $]-\varepsilon,\ep[\times]-a_2,a_2[\times\ldots\times]-a_N,a_N[$, and suppose that $R_{F}(\Omega)=\eps F^{o}(e_{1})$. Then
\[
\lim_{\ep\to 0^+}\frac{T_{p}(\O_\ep)}{\left(R_{F}(\O_\ep)\right)^{q}|\O_\ep|}=\frac{1}{q+1}\,.
\]
\end{prop}
The hypothesis $R_{F}(\Omega)=\eps F^{o}(e_{1})$ is not restrictive, in the sense that if it is not true we can choose a rotated $N$-rectangle where $R_{F}(\Omega)=\eps F^{o}(\nu)$ for some direction $\nu$, and use the remark below.
\begin{rem}
If $A\in\textrm{SO}(N)$ is a rotation matrix, then, denoting by $F_{A}(\xi)=F(A\xi)$, it holds that 
\[
(F_{A})^{o}(\xi)=(F^{o})_{A}(\xi),\quad \text{and then}\quad R_{F_{A^{T}}}(A\Omega)=R_{F}(\Omega)
\]
(see \cite{dgp} for the details). Hence, emphasizing the dependence on $F$ by denoting $T_{p}(\Omega)=T_{p,F}(\Omega)$, we have
\begin{align*}
T_{p,F}{(\Omega)}&=\max_{\varphi\in W_{0}^{1,p}(\Omega)}\frac{\ds\left(\int_{\Omega}|\varphi(x)|dx\right)^{p}}{\ds\int_{\Omega}F^{p}(\nabla \varphi(x))dx}
=
\max_{\varphi\in W_{0}^{1,p}(\Omega)}\frac{\ds\left(\int_{A\Omega}|\varphi(A^{T}x)|dx\right)^{p}}{\ds\int_{A\Omega}F_{A^{T}}^{p}(\nabla \varphi(A^{T}x))dx}\\
&=T_{p,F_{A^{T}}}(A\Omega)\ge \frac{|\Omega|}{q+1} R_{F_{A^{T}}}(A\Omega)^{q}=\frac{|\Omega|}{q+1}R_{F}(\Omega)^{q}.
\end{align*}
\end{rem}
\begin{proof}[Proof of Proposition \ref{rect}]
First of all, we observe that
\begin{equation*}
F^{o}(e_{1})=\frac{1}{F(e_{1})}.
\end{equation*}
Indeed, being $R_{F}(\Omega)=\eps F^{o}(e_{1})$, it holds that
\[
\nu_{\Omega}(\eps e_{1})=e_{1}=\frac{F^{o}_{\xi}(e_{1})}{|F^{o}_{\xi}(e_{1})|},
\]
where $\nu_{\Omega}(\eps e_{1})$ is the Euclidean outer normal vector to $\de\Omega$.
Hence by \eqref{seconda} and \eqref{prima}, we have
 \[
F(e_{1})= \frac{1}{|F_{\xi}^{o}(e_{1})|}=\frac{1}{F^{o}(e_{1})},
\]
where last equality follows by $F^{o}(e_{1})=F^{o}_{\xi}(e_{1})\cdot e_{1}=|F^{o}_{\xi}(e_{1})|$.

Let $\O_\ep=C_\ep\cup D_\ep$, where $C_\ep=]-\ep,\ep[\times]a_2+\ep,a_2-\ep[\times\ldots\times]-a_N+\ep,a_N-\ep[$, and $D_\ep= {\O_\ep}\setminus C_\ep$. Setting $x=(x_1,z)$ with $z\in\R^{N-1}$ and $a=(a_2,\ldots,a_N)$, we consider the function $\varphi_\ep$ defined by
\[
\begin{cases}
\varphi_\ep(x_1,z)=\ds\frac{\ep^q-x_1^q}{q}&\mbox{in }C_\ep\\
\varphi_\ep(x_1,z)=\ds\min\big\{|a-z|,|-a-z|\big\}\frac{\ep^q-x_1^q}{q\ep }&\mbox{in }D_\ep.
\end{cases}
\]


We can estimate the anisotropic $p$-torsional rigidity by using $\varphi_{\eps}$ as test function. We have:
\[
T_{p}(\O_\ep)^{p-1}\ge\frac{\ds\left(\int_{\O_\ep}\varphi_\ep\right)^p}{\ds\int_{\O_\ep} F^{p}(\nabla \varphi_\ep)dx}=\frac{\ds\left(\int_{C_\ep}\varphi_\ep+\int_{D_\ep}\varphi_\ep dx\right)^p}{\ds\int_{C_\ep}F^{p}(\nabla \varphi_\ep)dx+\int_{D_\ep}F^{p}(\nabla \varphi_\ep)dx}\;.
\]
We now compute
\[
\int_{C_\ep}\varphi_\ep\,dx=\int_{C_\ep}\frac{\ep^q-x_1^q}{q}\,dx=\frac{|C_\ep|\ep^q}{q+1}
\]
and
\[\int_{C_\ep}F^p(\nabla \varphi_\ep)\,dx=F^{p}(e_1)\int_{C_\ep} x_1^{q}\,dx=F^p(e_1)\frac{|C_\ep|\ep^q}{q+1}\;.
\]
We notice that both $\int_{D_\ep}\varphi_\ep\,dx$ and $\int_{D_\ep}F^{p}(\nabla \varphi_\ep)\,dx$ are negligible, since they go to zero as $\ep^{N+q-1}$. By recalling that
\[
\left(R_{F}(\Omega)\right)^q=\eps^{q}F^{o}(e_{1})^{q}=\frac{\ep^q}{F(e_1)^q}\;,
\]
we have 
\[
\frac{1}{q+1} \ge \lim_{\eps \to 0}\frac{T_{p}(\O_\ep)}{\left(R_{F}(\O_\ep)\right)^q|\O_\ep|}
\ge \frac{1}{q+1}
\]
which concludes the proof.
\end{proof}

\begin{rem}
We believe that the lower bound of $\Phi(\Omega)$ in \eqref{bounds} is not optimal. Actually, in the Euclidean setting, with $p=2$ our bound improves the analogous result of \cite{hlp}:
\[
\Phi(\Omega)\ge \frac{2}{N(N+2)}>\frac{1}{(N+1)^2}
\]

Moreover in \cite{hlp} the authors conjecture that for $F=\mathcal E$, $p=2$ and $N=2$ it holds 
\[
\Phi(\Omega) \ge \frac13,
\]
and 
\begin{equation}
\label{limite}
\Phi(\Omega_n)\to \frac 13,\quad\text{as }n\to\infty,
\end{equation}
where $\Omega_n$ is a sequence of isosceles triangles degenerating to a segment.

In the following example, for $F=\mathcal E$, $N=2$ and $p=2$, we find a sequence of degenerating triangles $\Omega_n$ such that \eqref{limite} holds. 
\end{rem}
\begin{example}\label{exHen}
Let
\[
N=2,\quad F(\xi)=\sqrt{\xi_1^2+\xi_2^2},\quad p=2.
\]
We want to show that there exists a sequence of thinning isosceles triangles $\tau_{a}$ of the plane such that
\begin{equation}
\label{henrot}
\Phi(\tau_{a})=\frac{T_2(\tau_{a})}{|\tau_{a}|M(\tau_{a})} \to \frac{1}{3} \quad\text{as}\quad a\to 0,
\end{equation}
where $T_{2}(\tau_{a})$ is the torsional rigidity of $\tau_{a}$, $M(\tau_{a})$ is the maximum of the torsion function in $\tau_{a}$ and $|\tau_{a}|$ is the area of the triangle.

First of all, we recall that by a result contained in \cite{fgl}, 
 for any sequence of isosceles triangles $\tau_{n}$ such that the ratio $\frac{R(\tau_{n})}{w(\tau_{n})}\to 0$, where $w(\tau_{n})$ is the width of $\tau_{n}$,  then
\[
\lim_{n \to \infty} \frac{T_2(\tau_{n})}{|\tau_{n}|} \frac{P^2(\tau_n)}{|\tau_n |^2}= \frac{2}{3}.
\]

Hence, recalling that in a triangle it holds that $R(\tau_{n})= \frac{2 |\tau_{n}|}{P(\tau_{n})}$  then
\[
\Phi(\tau_n)= \frac{T_2(\tau_{n})}{|\tau_{n}|}  \frac{P^2(\tau_n)}{|\tau_n |^2}\frac{R^2(\tau_n)}{4M(\tau_n)},
\]
the result is proved if we find a sequence of triangles  with vanishing ratio $R(\tau_n)/w(\tau_n)$ and such that $\frac{R^{2}(\tau_n)}{2M(\tau_n)}$ tends to $1$.


\begin{figure}[h]
\begin{center}
\begin{tikzpicture}[scale=1.5,line cap=round,line join=round,>=triangle 45,x=1.0cm,y=1.0cm]
\draw[-stealth,color=black] (-3.5,-0.453) -- (3.5,-0.453);
\draw[-stealth,color=black] (0.,-.6) -- (0.,.7);
\draw (0pt,-16pt) node[right] {\footnotesize $0$};
\draw [rotate around={0.:(0.,0.)}] (0.,0.) ellipse (0.8915105159222746cm and 0.453cm);
\draw[thick] (-3.2656304690483924,-0.453)-- (0.,0.5259598154674863);
\draw[thick] (0.,0.5259598154674863) node[anchor=south east]{\scriptsize $V_{1}$} -- (3.2656304690483924,-0.453);
\draw[thick] (3.2656304690483924,-0.453)-- (-3.2656304690483924,-0.453);
\draw (3.25,-0.457) node[anchor=north] {\scriptsize $V_{2}$};
\draw (-3.25,-0.457) node[anchor=north] {\scriptsize $V_{3}$};
\draw (-0.45,0.35) node[anchor=south east]{\scriptsize $(-a,y_a)$};
\draw (0.45,0.35) node[anchor=south west]{\scriptsize $(a,y_a)$};
\draw (-0.04,0.45) node[anchor=north west]{\scriptsize $2a$};
\end{tikzpicture}
\end{center}
\caption{}
\label{fig1}
\end{figure}
To this aim, let 
\[
\mathcal E_{a}=\left\{(x,y)\in\R^{2}\colon\frac{x^{2}}{1-a^{2}}+\frac{(y-a)^{2}}{a^{2}}=1\right\},
\]
and consider a point $(a,y_{a})$, with $y_{a}=a+a\sqrt{\frac{1-2a^{2}}{1-a^{2}}}$.
Let $\tau_{a}$ be the isosceles triangle constructed with one side on the $x$-axis and with each side tangent to the ellipse at the points $(0,0),$ $ (a,y_{a}),$ $ (-a,y_{a})$, as in Figure \ref{fig1}.

The vertices of the triangle are:
\[
V_{1}=\left(0, y_a+  \frac{a^{4}}{(1-a^{2})(y_a-a)}\right),\quad V_{2}=\left( a+ \frac{y_a(y_a-a)}{a^{3}}(1-a^{2}),0\right),\quad V_{3}=-V_{2}.
\]
Let us observe that $V_{1}\to (0,0)$ as $a\to 0$, while the first coordinate of $V_{2}$ diverges.

Then, denoting by $A(\tau_{a})$ and $P(\tau_{a})$ respectively the area and the perimeter of $\tau_{a}$, and by
\[
h= y_a+ \frac{a^{4}}{(1-a^{2})(y_a-a)},\quad 
\frac b 2=a+ \frac{y_a(y_a-a)}{a^{3}}(1-a^{2}),
\]
we have:
\[
R(\tau_{a})= \frac{2 |\tau_{a}|}{P(\tau_{a})} = \frac{bh}{b+\sqrt{2 h^{2}+{b^{2}}}}.
\]
Now, being $\mathcal E_{a}\subset \tau_{a}$, by the comparison principle and \eqref{stima_max} it holds that
\[
M(\mathcal E_{a})\le M(\tau_{a}) \le \frac{R^{2}(\tau_{a})}{2},\qquad M(\mathcal E_{a})=\frac{a^{2}(1-a^{2})}{2}, 
\]
where the maximum of the torsion function on $\mathcal E_{a}$ follows by a direct computation.
Then, being $h=2a+o(a^{2})$ and $b\to +\infty$ as $a\to 0$, we have

\begin{multline*}
1\le\frac{R(\tau_{a})^{2}}{2 M(\mathcal E_{a})}= \left(\frac{b}{b+\sqrt{2h^{2}+b^{2}}}\right)^{2}\frac{h^{2}}{a^{2}(1-a^{2})}=\left(\frac{1}{1+\sqrt{2\frac{h^{2}}{b^{2}}+1}}\right)^{2}\cdot\frac{h^{2}}{a^{2}(1-a^{2})}\longrightarrow 1\\ \text{ as }a\to 0
\end{multline*}
and \eqref{henrot} is proved.

We explicitly observe that, from the above computations, it holds 
\[
\frac{R^{2}(\mathcal E_{a})}{2M(\mathcal E_{a})}\to 1 \qquad\text{as }a\to 0.
\]
\end{example}

\section*{Acnowledgements}

This work has been partially supported by the FIRB 2013 project ``Geometrical and qualitative aspects of PDE's'' and by GNAMPA of INdAM.

\end{document}